\documentclass[12pt,oneside]{amsart}

\usepackage[normalem]{ulem} 
\usepackage{latexsym,amssymb,amsmath,graphicx,hyperref,subfig,caption,verbatim}
\usepackage{tikz}
\usetikzlibrary{shapes}
\usetikzlibrary{arrows}
\usepackage{fp}
\usepackage{rotating}
\usepackage{longtable}
\usepackage{mathtools}
\usepackage{centernot}

\numberwithin{equation}{section}
\newtheorem{theorem}{Theorem}[section]
\newtheorem{lemma}[theorem]{Lemma}

\theoremstyle{definition}

\newtheorem{remark}[theorem]{Remark}

\newcommand{\PIVAL}{3.14159265358979323846264338} 
\newcounter{i} 
\newcommand{\Circulant}[2] { 
	\begin{tikzpicture}
	\setcounter{i}{0}
	\whiledo{\value{i}<#1}{ 
		\FPmul\tempA{2}{\thei} 
		\FPdiv\tempB{\PIVAL}{#1} 
		\FPmul\tempC{\tempA}{\tempB} 
		\FPcos\varX{\tempC} 
		\FPsin\varY{\tempC} 
		\stepcounter{i} 
		\FPround\varX{\varX}{3}
		\FPround\varY{\varY}{3}
		\node (\thei) at (\varX,\varY)[place]{ }; 
		\foreach \x in {#2} { 
			\pgfmathparse{mod(\x+\thei,#1)} 
			\let\tempB\pgfmathresult
			\pgfmathparse{mod(\thei-\x,#1)} 
			\let\tempA\pgfmathresult
			\ifthenelse{\lengthtest{\tempA pt < 1 pt}}{\FPadd\tempA{\tempA}{#1}}{}
			\ifthenelse{\lengthtest{\tempB pt < 1 pt}}{\FPadd\tempB{\tempB}{#1}}{}
			\ifthenelse{\lengthtest{\tempA pt > \thei pt}}{}{\ifthenelse{\thei = \tempA}{}{\draw [] (\thei) to (\tempA)}};
			\ifthenelse{\lengthtest{\tempB pt > \thei pt}}{}{\ifthenelse{\thei = \tempB}{}{\draw [] (\thei) to (\tempB)}};
		}
	}
	\end{tikzpicture}
}

\begin{document}

\tikzstyle{place}=[draw,circle,minimum size=0.5mm,inner sep=1pt,outer sep=-1.1pt,fill=black]

 
\title[Regularity of circulant graphs]{The regularity of some
  families of circulant graphs}
\thanks{Submitted Version: June 14, 2019}
      
\author{Miguel Eduardo Uribe-Paczka}
\address{Departamento de Matem\'aticas,
  Escuela Superior de F\'isica y Matem\'ticas, 
  Instituto Polit\'ecnico Nacional, 07300 Mexico City}
\email{muribep1700@alumno.ipn.mx}
 
\author{Adam Van Tuyl}
\address{Department of Mathematics and Statistics\\
  McMaster University, Hamilton, ON, L8S 4L8}
\email{vantuyl@math.mcmaster.ca}


\begin{abstract}
We compute the Castelnuovo-Mumford regularity of the edge ideals
of two families of circulant graphs, which includes all cubic circulant graphs.
A feature of our approach is to combine bounds on the regularity,
the projective dimension, and the reduced Euler characteristic
to derive an exact value for the regularity.
\end{abstract}

\maketitle


\section{Introduction} 

Let $G$ be any finite simple graph with vertex set $V(G) = \{x_1,\ldots,x_n\}$
and edge set $E(G)$, where simple means no loops or
multiple edges.
The \textbf{edge ideal} of $G$ is the ideal
$I(G)=
\left\langle  x_{i}x_{j} \mid \{x_{i},x_{j} \}\in E(G) \right\rangle$
in the polynomial ring $R = k[x_1,\ldots,x_n]$ ($k$ is any field).
Describing the
dictionary between the graph theoretic properties of $G$
and the algebraic properties of $I(G)$ or $R/I(G)$ is an 
active area of research, e.g., see
\cite{MV,V}.

Relating the homological
invariants of $I(G)$ and the graph theoretic invariants of $G$
has proven to be a fruitful approach to building this
dictionary.  Recall that the  
\textbf{minimal graded free resolution} of $I(G) \subseteq R$
is a long exact sequence of the form
\[
0 \rightarrow \bigoplus_{j} R(-j)^{\beta_{l,j}(I(G))} 
\rightarrow   \bigoplus_{j} R(-j)^{\beta_{l-1,j}(I(G))} \rightarrow \cdots 
\rightarrow   \bigoplus_{j} R(-j)^{\beta_{0,j}(I(G))} \rightarrow 
I(G) \rightarrow 0
\]
where $l \leq n$ and $R(-j)$ is the free $R$-module obtained by shifting 
the degrees of $R$ by $j$ (i.e., $R(-j)_a = R_{a-j}$). 
The number $\beta_{i,j}(I(G))$, the $i,j$-th 
\textbf{graded Betti number} of $I(G)$, equals the number of
minimal generators of degree $j$ in the $i$-th syzygy module of 
$I(G)$.  Two invariants that measure the ``size''
of the resolution are the \textbf{(Castelnuovo-Mumford) regularity}
and the \textbf{projective
dimension}, defined as
\begin{eqnarray*}
{\rm reg}(I(G)) & = & \max\{ j-i \mid \beta_{i,j}(I(G))\neq 0 \}, ~~\mbox{and}~~\\
{\rm pd}(I(G))  & = & \max\{ i \mid \beta_{i,j}(I(G))\neq 0 ~~\mbox{for some $j$}\}.
\end{eqnarray*}
One wishes to relate the numbers $\beta_{i,j}(I(G))$ to the
invariants of $G$; e.g., see the survey of H\`a \cite{Ha} which focuses
on describing ${\rm reg}(I(G))$ in terms of the invariants of $G$.

In this note we give explicit formulas for ${\rm reg}(I(G))$ for the
edge ideals of two infinite families of circulant graphs.
Our results complement previous work on the algebraic and
combinatorial topological properties of circulant
graphs (e.g, \cite{EVMVT,M,Ri,RiRo,RR,R,VMVTW}).
Fix an integer $n \geq 1$ and a subset $S \subseteq \{1,\ldots,\lfloor
\frac{n}{2} \rfloor \}$.  The \textbf{circulant graph}
$C_n(S)$ is the graph on the vertex set $\{x_{1},\ldots,x_{n}\}$
such that $\{x_i,x_j\} \in E(C_n(S))$ if and only if 
$|i-j|$ or $n-|i-j| \in S$. To simplify notation, we write $C_n(a_1,\ldots,a_t)$ instead of
$C_n(\{a_1,\ldots,a_t\})$. As an example, the graph $C_{10}(1,3)$ is drawn
in Figure \ref{circulantex}.
\begin{figure}[h!]
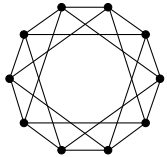

\Circulant{10}{1,3} 
\caption{The circulant $C_{10}(1,3)$}\label{circulantex}
\end{figure}

When $S = \{1,\ldots,\lfloor \frac{n}{2} \rfloor\}$, then
$C_n(S) \cong K_n$, the clique on $n$ vertices.
On the other hand, if $S = \{1\}$,
then $C_n(1) \cong C_n$, the cycle on $n$ vertices.   For both of these
families, the regularity of their edge ideals are known.
Specifically, the ideal
$I(K_n)$ has a linear
resolution by Fr\"oberg's Theorem \cite{F}, so ${\rm reg}(I(K_n))=2$.
The value of ${\rm reg}(I(C_n))$
can be deduced from work of Jacques
\cite[Theorem 7.6.28]{J}.  One can view
these circulant graphs as ``extremal'' cases in the sense that $|S|$ is either
as large, or small, as possible.

Our motivation is to understand the
next open cases.   In particular, generalizing the case of $K_n$, 
we compute ${\rm reg}(I(C_n(S))$ when $S =
\{1,\ldots,\widehat{j},\ldots,\lfloor \frac{n}{2} \rfloor\}$ for
any $1 \leq j \leq \lfloor \frac{n}{2} \rfloor$
(Theorem \ref{maintheorem1}).  For most $j$, the regularity
follows from Fr\"oberg's Theorem and a result
of Nevo \cite{E}. To generalize the
case of $C_n$ (a circulant graph where every vertex has degree two),
we compute the regularity of the edge ideal of 
any cubic (every vertex has degree three) circulant graph, that is, $G = C_{2n}(a,n)$ with
$1 \leq a \leq n$ (Theorem \ref{maintheorem2}).  Our proof
of Theorem \ref{maintheorem2} requires a new technique
to compute ${\rm reg}(I)$ for a square-free monomial ideal.  Specifically,
we show how to use partial information about
${\rm reg}(I)$, ${\rm pd}(I)$, and the reduced
Euler characteristic of the simplicial complex associated with $I$,
to determine ${\rm reg}(I)$ exactly (see
Theorem \ref{new-reg-result}).  We believe
this result to be of independent interest.

Our paper is structured as follows. In Section~\ref{Background} we
recall the relevant background regarding graph theory and commutative algebra,
along with our new result on the regularity of square-free
monomial ideals. In Section~\ref{reg-hat-j} we compute the
regularity  of $I(G)$ for the family of graphs
$G =C_{n}(1,\ldots,\hat{j},\ldots,\lfloor \frac{n}{2} \rfloor)$.
In Section~\ref{regularity-two-n-a-n}, we give an explicit formula
for the regularity of edge ideals of cubic circulant graphs.

\noindent
    {\bf Acknowledgements.} The authors thank Federico Galetto and Andrew Nicas
    for their comments and suggestions. Computation using
    {\it Macaulay2} \cite{Mt}  inspired some of our results.
    The first author thanks CONACYT for
    financial support, and the second author acknowledges the financial
    support of NSERC RGPIN-2019-05412.


\section{Background} \label{Background}

We review the relevant background from
graph theory and commutative algebra.  In addition,
we give a new result on the regularity of square-free monomial ideals.

\subsection{Graph theory preliminaries}
Let $G = (V(G),E(G))$ denote a finite simple graph.  We abuse
notation and write $xy$ for the edge $\{x,y\} \in E(G)$.  The
{\bf complement} of $G$, denoted $G^c$, is the graph $(V(G^c),E(G^c))$
where $V(G^c) = V(G)$ and $E(G^c) = \{xy ~|~
xy \not\in E(G)\}$.  The {\bf neighbours} of
$x \in V(G)$ is the set $N(x) = \{y ~|~ xy \in E(G)\}$.  
The {\bf closed neighbourhood}
of $x$ is $N[x] = N(x) \cup \{x\}$.  The {\bf degree} of $x$ is
$\deg(x) = |N(x)|$.

A graph $H = (V(H),E(H))$ is a {\bf subgraph} of $G$ if
$V(H) \subseteq V(G)$ and $E(H) \subseteq E(G)$.  Given a subset
$W \subseteq V(G)$, the {\bf induced subgraph} of $G$ on $W$ is
the graph $G_W = (W,E(G_W))$ where 
$E(G_W) = \{ xy \in E(G) ~|~ \{x,y\} \subseteq W\}$.  Notice that 
an induced subgraph is a subgraph of $G$, but not every subgraph
of $G$ is an induced subgraph.

An {\bf $n$-cycle}, denoted $C_n$, is the graph with
$V(C_n) = \{x_1,\ldots,x_n\}$
and edges $E(C_n) = \{x_1x_2,x_2x_3,\ldots,x_{n-1}x_n,x_nx_1\}$.
A graph $G$ has a {\bf cycle} of length $n$ if $G$ has a subgraph
of the form $C_n$.  A graph is a 
\textbf{chordal graph} if $G$ has no induced graph of the 
form $C_n$ with $n \geq 4$.
A graph $G$ is \textbf{co-chordal} if $G^c$ is chordal.
The \textbf{co-chordal number} of $G$, denoted ${\rm co\mbox{-}chord}(G)$,
is the smallest number of subgraphs of $G$ such that
$G = G_1 \cup \cdots \cup G_s$ and each $G_i^c$ is a chordal graph. 

A \textbf{claw} is the graph with $V(G) =\{x_1,x_2,x_3,x_4\}$
with edges $E(G) = \{x_1x_2,x_1x_3,x_1x_4\}$.
A graph is {\bf claw-free} if no induced subgraph of the graph is a claw.
A graph $G$ is {\bf gap-free}
if no induced subgraph of $G^c$ is a $C_4$.  Finally, the {\bf complete graph}
$K_n$ is the graph with $V(K_n) = \{x_1,\ldots,x_n\}$ and 
$E(K_n) = \{x_ix_j ~|~ 1 \leq i < j \leq n\}$.

\subsection{Algebraic preliminaries}
We recall some facts about the regularity of $I(G)$.  Note that for
any homogeneous ideal, ${\rm reg}(I) = {\rm reg}(R/I) + 1$.

We collect together a number of useful
results on the regularity of edge ideals.

\begin{theorem}\label{reg-results}
  Let $G$ be a finite simple graph.  Then
  \begin{enumerate}
   \item[$(i)$] if $G = H \cup K$, with $H$ and $K$ disjoint, then
    \[{\rm reg}(R/I(G)) = {\rm reg}(R/I(H)) + {\rm reg}(R/I(K)).\]
   \item[$(ii)$]  ${\rm reg}(I(G)) = 2$
     if and only if $G^c$ is a chordal graph.
  \item[$(iii$)] ${\rm reg}(I(G)) \leq {\rm co\mbox{-}chord}(G) +1$.
  \item[$(iv)$] if $G$ is gap-free and claw-free, then ${\rm reg}(I(G)) \leq 3$.
  \item[$(v)$] if $x \in V(G)$, then 
${\rm reg}(I(G))\in{ \{ {\rm reg}(I( G \setminus N_{G}[x]))+1, {\rm reg}(I( G \setminus x)) \} }.$
 \end{enumerate}
\end{theorem}

\begin{proof}  For $(i)$, see Woodroofe \cite[Lemma 8]{W}.  Statement $(ii)$ is
 Fr\"oberg's Theorem \cite[Theorem 1]{F}.
  Woodroofe \cite[Theorem 1]{W} first proved $(iii)$.  Nevo first
  proved $(iv)$ in \cite[Theorem 5.1]{E}.  For $(v)$, see
  Dao, Huneke, and Schweig \cite[Lemma 3.1]{DHS}.
\end{proof}

We require a result of Kalai and Meshulam \cite{KM}
that has been specialized to edge ideals.

\begin{theorem}\cite[Theorems 1.4 and 1.5]{KM} \label{KM}
Let $G$ be a finite simple graph, and suppose $H$ and $K$ are subgraphs
such that $G = H \cup K$.  Then,
\begin{enumerate}
\item[$(i)$] ${\rm reg}(R/I(G)) \leq {\rm reg}(R/I(H)) + {\rm reg}(R/I(K))$,
  and 
\item[$(ii)$] ${\rm pd}(I(G)) \leq {\rm pd}(I(H)) + {\rm pd}(I(K)) + 1$.
\end{enumerate}
\end{theorem}

We now introduced a new result on the regularity of edge ideals.
In fact, because our result holds for all square-free monomial ideals, 
we present the more general case.   

Recall the following facts about simplicial complexes.
A \textbf{simplicial complex} $\Delta$ on a vertex set 
$V=\{x_{1},\ldots,x_{n}\}$ is a set of subsets of $V$ that 
satisfies: $(i)$ if $F\in{\Delta}$ and $G \subseteq F$, then $G\in{\Delta}$, 
and $(ii)$ for each $i\in{\{1,\ldots,n\}}$,
$\{x_{i}\}\in \Delta$. Note that condition $(i)$ implies that 
$\emptyset \in{\Delta}$. The elements of $\Delta$ are called 
its \textbf{faces}. For any $W \subseteq V$, the 
restriction of $\Delta$ to $W$ is the simplicial complex $\Delta_{W}=\{ F\in{\Delta} \mid F\subseteq W\}$.

The \textbf{dimension} of a face $F\in{\Delta}$ is given by 
${\rm dim}(F)=\left|F\right|-1$. The \textbf{dimension} of a
simplicial complex, denoted by ${\rm dim}(\Delta)$, is the 
maximum dimension of all its faces. Let $f_i$ be the number of faces of
$\Delta$ of dimension $i$, with the convention that $f_{-1}=1$. 
If ${\rm dim}(\Delta)=D$, then the $f$-\textbf{vector} of $\Delta$
is the $(D+2)$-tuple $f(\Delta)=(f_{-1},f_0,\ldots,f_D)$.

Given any simplicial complex $\Delta$ on $V$, 
associate with $\Delta$ a monomial ideal $I_{\Delta}$ in
the polynomial ring $R=k[x_{1},\ldots,x_n]$ (with $k$ a field) as follows:
\[
I_{\Delta}=\left\langle  x_{j_{1}}x_{j_{2}} \cdots x_{j_{r}} \mid \{ x_{j_{1}},\ldots,x_{j_{r}} \}\notin{\Delta}  \right\rangle. 
\]
The ideal $I_{\Delta}$ is the \textbf{Stanley-Reisner ideal} of $\Delta$.  This
construction can be reversed.  Given a square-free monomial ideal $I$ of
$R$,  the simplicial complex associated with $I$ is 
\[\Delta(I) = \left\{\{x_{i_1},\ldots,x_{i_r}\} ~|~ 
\mbox{the square-free monomial} ~x_{i_1}\cdots x_{i_r}\not\in I \right\}.\]

Given a square-free monomial ideal $I$, 
Hochster's Formula relates the Betti numbers of $I$ to the reduced
simplicial homology of $\Delta(I)$.  See \cite[Section 6.2]{V}
for more background 
on $\widetilde{H}_j(\Gamma;k)$, the $j$-th reduced simplicial homology
group of a simplicial complex $\Gamma$.

\begin{theorem}{\rm(Hochster's Formula)}
\label{Hochsters-Formula}
Let $I \subseteq R = k[x_1,\ldots,x_n]$ be a square-free monomial
ideal, and set $\Delta = \Delta(I)$.  Then, for all $i,j \geq 0$, 
\[
\beta_{i,j}(I)= \sum_{\left| W \right|=j,~ W \subseteq V} 
\dim_k \widetilde{H}_{j-i-2}(\Delta_W ; k).
\]
\end{theorem}

Given a simplicial complex $\Delta$ of dimension $D$, 
the dimensions of the homology groups
$\widetilde{H}_{i}(\Delta ; k)$ are related to the
$f$-vector $f(\Delta)$ via the \textbf{reduced Euler characteristic}:
\begin{equation}
\label{def-Euler-characteristic}
\widetilde{\chi}(\Delta)= 
\sum_{i=-1}^D (-1)^{i} \dim_k \tilde{H}_{i}(\Delta ; k)=
\sum_{i=-1}^D (-1)^{i}f_i.
\end{equation}
Note that the 
reduced Euler characteristic is normally defined to be equal to one
of the two sums, and then one proves the two sums are equal
(e.g., see \cite[Section 6.2]{V}).

Our new result on the regularity of square-free
monomial ideals allows us to determine ${\rm reg}(I)$
exactly if we have enough partial information on the
regularity, projective dimension, and the reduced Euler characteristic.

\begin{theorem}\label{new-reg-result}
Let $I$ be a square-free monomial ideal of $R = k[x_1,\ldots,x_n]$
with associated simplicial complex $\Delta = \Delta(I)$.  
\begin{enumerate}
\item[$(i)$] Suppose that ${\rm reg}(I) \leq r$ and ${\rm pd}(I) \leq
n-r+1$.
\begin{enumerate}
\item[$(a)$] If $r$ is even and $\widetilde{\chi}(\Delta) > 0$,
then ${\rm reg}(I) =r$.
\item[$(b)$] If $r$ is odd and $\widetilde{\chi}(\Delta) < 0$,
then ${\rm reg}(I) =r$.
\end{enumerate}
\item[$(ii)$] Suppose that ${\rm reg}(I) \leq r$ and ${\rm pd}(I) \leq
n-r$.  If $\widetilde{\chi}(\Delta) \neq 0$, then ${\rm reg}(I) = r$.
\end{enumerate}

\end{theorem}

\begin{proof}
By Hochster's Formula (Theorem \ref{Hochsters-Formula}), note that
$\beta_{a,n}(I) = \dim_k \widetilde{H}_{n-a-2}(\Delta;k)$ for all
$a \geq 0$ 
since the only subset $W \subseteq V$ with $|W| = n$ is $V$.

$(i)$  If ${\rm reg}(I) \leq r$ and ${\rm pd}(I) \leq n-r+1$
we have $\beta_{a,n}(I) = 0$ for all $a\leq n-r-1$ and $\beta_{a,n}(I) = 0$
for all $a \geq n-r+2$.  Consequently, among all the
graded Betti numbers of form $\beta_{a,n}(I)$ as $a$ varies,
all but $\beta_{n-r,n}(I) = \dim_k \widetilde{H}_{r-2}(\Delta;k) $ and
$\beta_{n-r+1,n}(I) = \dim_k \widetilde{H}_{r-3}(\Delta;k) $
 may be non-zero.  Thus by \eqref{def-Euler-characteristic}
\begin{eqnarray*}
\widetilde{\chi}(\Delta)& =& 
(-1)^{r-2}  \dim_k \widetilde{H}_{r-2}(\Delta;k) + 
(-1)^{r-3} \dim_k \widetilde{H}_{r-3}(\Delta;k).
\end{eqnarray*}
If we now suppose that $r$ is even and $\widetilde{\chi}(\Delta) > 0$,
the above expression implies 
\[\dim_k \widetilde{H}_{r-2}(\Delta;k) - \dim_k \widetilde{H}_{r-3}(\Delta;k) > 
0, \]
and thus $\beta_{n-r,r}(I) = \dim_k \widetilde{H}_{r-2}(\Delta;k) \neq 0$.
As a consequence, ${\rm reg}(I) = r$, thus proving $(a)$.  Similarly,
if $r$ is odd and $\widetilde{\chi}(\Delta) < 0$, this again forces
$\beta_{n-r,r}(I) = \dim_k \widetilde{H}_{r-2}(\Delta;k) \neq 0$, thus
proving $(b)$.

$(ii)$  Similar to part $(i)$, the hypotheses on the regularity and
projective dimension imply that $\widetilde{\chi}(\Delta) = 
\beta_{n-r,n}(I) = (-1)^{r-2}\dim_k \widetilde{H}_{r-2}(\Delta;k)$.  
So, if $\widetilde{\chi}(\Delta) \neq 0$, then
$\beta_{n-r,n}(I) \neq 0$, which implies ${\rm reg}(I) = r$.
\end{proof}

\begin{remark}
There is a similar result to Theorem \ref{new-reg-result} for the 
projective dimension of $I$.  In particular, under the assumptions of
$(i)$ and if $r$ is even and $\widetilde{\chi}(\Delta) < 0$, or
if $r$ is odd and $\widetilde{\chi}(\Delta) > 0$, then
the proof of Theorem \ref{new-reg-result} shows that
${\rm pd}(I) = n-r+1.$  Under the assumptions of $(ii)$, then
${\rm pd}(I) = n-r$.
\end{remark}

We will apply Theorem \ref{new-reg-result}
to compute the regularity of cubic circulant graphs (see
Theorem \ref{maintheorem2}).
We will also require the following terminology and results
which relates the reduced Euler characteristic to the independence
polynomial of a graph.

A subset $W \subseteq V(G)$ is an \textbf{independent set} if for all
$e\in{E(G)}$, $e \nsubseteq W$.  The \textbf{independence complex} of $G$ is the set of all independent sets:
\[
{\rm Ind}(G)=\{ W \mid W \mbox{ is an independent set of } V(G)\}. 
\]
Note that ${\rm Ind}(G) = \Delta_{I(G)}$, the simplicial complex 
associated with the edge ideal $I(G)$.

The \textbf{independence polynomial} of a graph $G$  is defined as
\[
I(G,x)= \sum_{r=0}^\alpha i_{r}x^{r},
\]
where $i_r$ is the number of independent sets of
cardinality $r$. 
Note that 
$(i_0,i_1,\ldots,i_\alpha)=(f_{-1},f_0,\ldots,f_{\alpha-1})$ is the
 $f$-vector of ${\rm Ind}(G)$. Since 
$\widetilde{\chi}({\rm Ind}(G))=\sum_{i=-1}^{\alpha-1} (-1)^{i}f_i$, we get: 
\begin{equation}
\label{eq-Euler-independence}
\widetilde{\chi}({\rm Ind}(G))=-I(G,-1). 
\end{equation}
Thus, the value of $\widetilde{\chi}({\rm Ind}(G))$ can be extracted
from the independence polynomial $I(G,x)$.


\section{The regularity of the edge ideals of 
$C_{n}(1,\ldots,\widehat{j},\ldots,\lfloor \frac{n}{2} \rfloor)$
}
\label{reg-hat-j}

In this section we compute the regularity of the edge ideal
of the circulant graph $G = C_n(S)$ with $S = \{1,\ldots,\widehat{j},\ldots,
\lfloor \frac{n}{2} \rfloor\}$ for any
$j \in \{1,\ldots,\lfloor \frac{n}{2} \rfloor\}$.

We begin with the observation that the complement of $G$ is also
a circulant graph, and in particular, $G^c = C_n(j)$.  Furthermore,
we have the following structure result.

\begin{lemma}\label{complement}
  Let $H = C_n(j)$ with
  $1 \leq j \leq \left\lfloor \frac{n}{2}
  \right\rfloor$,
  and set $d = {\rm gcd}(j,n)$. 
  Then $H$ is the union of $d$ disjoint cycles of length $\frac{n}{d}$. 
  Furthermore, $H$ is a chordal graph if and only if
  $n=2j$ or $n=3d$. 
\end{lemma}

\begin{proof}
  Label the vertices of $H$ as $\{0,1,\ldots,n-1\}$, and set
  $d = {\rm gcd}(j,n)$.  For each $0 \leq i < d$, the
  induced graph on the vertices $\{i,j+i,2j+i,\ldots,(d-1)j+i\}$
  is a cycle of length $\frac{n}{d}$, thus proving the first statement (if
  $\frac{n}{d} = 2$, then $H$ consists of disjoint edges).
  For the second statement, if $n=3d$, then $H$ is the disjoint union of
  three cycles, and thus chordal.  If $n=2j$, then $H$ consists of
  $j$ disjoint edges, and consequently, is chordal.  Otherwise,
  $\frac{n}{d} \geq 4$, and so $H$ is not chordal.
\end{proof}

\begin{lemma} \label{reg-general}
  Let $G = C_{n}(1,\ldots,\widehat{j},\ldots,\lfloor \frac{n}{2} \rfloor)$,
  and $d = {\rm gcd}(j,n)$.  If $\frac{n}{d} \geq 5$, then
  $G$ is claw-free.
\end{lemma}

\begin{proof}
Suppose that $G$ has an induced subgraph $H$ on $\{z_1,z_2,z_3,z_4\}
\subseteq V(G)$ that is a claw.  Then $H^{c}$ is an induced subgraph 
of $G^{c}$ of the form: 
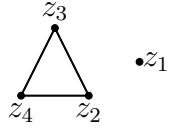
\begin{center}
\begin{tikzpicture}
[scale=.45]
\draw[thick] (0,0) --(2,0)--(1,2)--(0,0);
\draw [fill] (0,0) circle [radius=0.1];
\draw [fill] (2,0) circle [radius=0.1];
\draw [fill] (1,2) circle [radius=0.1];
\draw [fill] (3.5,1) circle [radius=0.1];

\node at (0,-0.5) {$z_{4}$};
\node at (2,-0.5) {$z_{2}$};
\node at (1,2.5) {$z_{3}$};
\node at (4,1) {$z_{1}$};
\end{tikzpicture}
\end{center}
But by Lemma \ref{complement}, the induced cycles of $G^c$ 
have length $\frac{n}{d} \geq 5$.
Thus $G$ is claw-free.
\end{proof}

We now come to the main theorem of this section.

\begin{theorem}\label{maintheorem1}
  Let $G =C_{n}(1,\ldots,\widehat{j},\ldots,\lfloor \frac{n}{2} \rfloor)$.
  If $d={\rm gcd}(j,n)$, then
\[
  {\rm reg}(I(G)) =
  \begin{cases}
  2 & \mbox{$n=2j$ or $n=3d$} \\
  3 & \mbox{otherwise.}
\end{cases}
\]
\end{theorem}

\begin{proof}
  Consider $G^c = C_n(j)$.  By Lemma \ref{complement}, $G^c$ consists
  of induced cycles of size $k = \frac{n}{d}$.  Because
  $1 \leq j \leq \lfloor \frac{n}{2} \rfloor$, we have $2 \leq k \leq n$.
  If $k=2$ or $3$, i.e., if $n=2j$ or $n=3d$,
  Lemma \ref{complement} and Theorem \ref{reg-results} $(ii)$
  combine to give ${\rm reg}(I(G)) = 2$.  If $k \geq 5$, then
  Lemmas \ref{complement} and \ref{reg-general} imply that
  $G$ is gap-free and claw-free (but not chordal), and  so
  Theorem \ref{reg-results} $(iv)$ implies ${\rm reg}(I(G)) = 3$.
  
  To compete the proof, we need to consider the case $k=4$, i.e., 
  $G = C_{4j}(1,\ldots,\widehat{j},\ldots,2j)$.
  By Lemma \ref{complement}, 
  $G^c$ is $j$ disjoint copies of $C_4$, and thus
  Theorem \ref{reg-results} $(ii)$ gives ${\rm reg}(I(G)) \geq 3$.
  To prove that ${\rm reg}(I(G)) = 3$, we 
  show ${\rm co\mbox{-}chord}(G)= 2$, and apply
  Theorem \ref{reg-results} $(iii)$.
  
  Label the vertices of $G$ as $0,1,\ldots, 4j-1$, and let 
\begin{eqnarray*}
V_1 &= &\{0,1,2,\ldots,j-1,2j,2j+1,\ldots,3j-1\} ~~\mbox{and}~~\\
V_2 &= &\{j,j+1,\ldots,2j-1,3j,3j+1,\ldots,4j-1\}.
\end{eqnarray*}
Observe that the induced graph on $V_1$ (and $V_2)$ is the complete
graph $K_{2j}$.  

Let $G_1$ be the graph with $V(G_1) = V(G)$ and edge set
$E(G_1) = (E(C_{4j}(1,\ldots,j-1)) \cup E(G_{V_1})) \setminus E(G_{V_2})$.
Similarly, we let $G_2$ be the graph with $V(G_2) = V(G)$ and edge
set $E(G_2) =  (E(C_{4j}(j+1,\ldots,2j)) \cup E(G_{V_2})) \setminus E(G_{V_1})$.

We now claim that $G = G_1 \cup G_2$, and furthermore, both
$G_1^c$ and $G_2^c$ are chordal, and consequently,
${\rm co\mbox{-}chord}(G) =2$.   The fact that $G = G_1 \cup G_2$
follows from the fact that 
\begin{eqnarray*}
E(G_1) \cup E(G_2)& =& E(C_{4j}(1,\ldots,j-1)) \cup E(C_{4j}(j+1,\ldots,2j))\\
&=& E(G_{4j}(1,\ldots,\widehat{j},\ldots,2j)).
\end{eqnarray*}

To show that $G_1^c$ is chordal, first note that the induced graph
on $V_1$, that is, $(G_1)_{V_1}$ is the complete graph $K_{2j}$.  
In addition, the vertices $V_2$ form an independent set of $G_1$.  
To see why, note that if $a,b \in V_2$ are such that $ab \in E(G)$,
then $ab \in E(G_{V_2})$.  But by the construction of $E(G_1)$, none
of the edges of $E(G_{V_2})$ belong to $E(G_1)$.  So $ab \not\in
E(G_1)$, and thus $V_2$ is an independent set in $G_1$.  

The above observations therefore imply that in $G_1^c$, the vertices
of $V_1$ form an independent set, and $(G_1^c)_{V_2}$ is the clique
$K_{2j}$.  To show that $G_1^c$ is chordal, suppose that $G_1^c$ has a induced
cycle of length $t \geq 4$ on $\{v_1,v_2,v_3,\ldots,v_t\}$.  Since
the induced graph on $(G_1^c)_{V_2}$ is a clique, at most two of
the vertices of $\{v_1,v_2,\ldots,v_t\}$ can belong to $V_2$.  Indeed,
if there were at least three $v_i,v_j,v_k \in \{v_1,v_2,\ldots,v_t\} \cap V_2$,
then the induced graph on these vertices is a three cycle, contradicting
the fact that  $\{v_1,v_2,\ldots,v_t\}$ is minimal induced cycle of
length $t \geq 4$.  But then at least $t-2 \geq 2$ vertices of 
$\{v_1,v_2,\ldots,v_t\}$ must belong to $V_1$, and in particular,
at least two of them are adjacent.  But this cannot happen since 
the vertices of $V_1$ are independent in $G_1^c$.  Thus, $G_1^c$ must
be a chordal graph.

The proof that $G_2^c$ is chordal is similar.  Note that the vertices
of $V_2$ are an independent set, and $(G_2^c)_{V_1}$ is the clique
$K_{2j}$.  The proof now proceeds as above.
\end{proof}


\section{Cubic circulant graphs}
\label{regularity-two-n-a-n}

We  now compute the regularity of the edge ideals of {\bf cubic circulant
graphs}, that is, a circulant graph 
where  every vertex has degree three.  Cubic circulant graphs
have the form $G = C_{2n}(a,n)$ with $1 \leq a \leq n$.  The main
result of this section can also be viewed as an application
of Theorem \ref{new-reg-result} to compute the regularity of a square-free
monomial ideal.

We begin with a structural result for cubic circulants
due to Davis and Domke.

\begin{theorem}\cite{DDG}\label{isomorphic-a-n}
Let $1 \leq a < n$  and $t={\rm gcd}(2n,a)$. 
\begin{enumerate}
\item[$(a)$] If $\frac{2n}{t}$ is even, then $C_{2n}(a,n)$ is isomorphic to $t$ copies of $C_{\frac{2n}{t}}(1,\frac{n}{t})$. 
\item[$(b)$] If $\frac{2n}{t}$ is odd, then $C_{2n}(a,n)$ is isomorphic to $\frac{t}{2}$ copies of $C_{\frac{4n}{t}}(2,\frac{2n}{t})$. 
\end{enumerate}
\end{theorem}
\noindent
By combining the previous theorem with Theorem \ref{reg-results} $(i)$, to
compute the regularity of the edge ideal of any cubic circulant graph, it
suffices to compute the regularity of the edge ideals of $C_{2n}(1,n)$ and
$C_{2n}(2,n)$.  Observe that $n$ must be odd in $C_{2n}(2,n)$.  It will be convenient to use the representation
and labelling  of these two graphs in Figure \ref{cubic-picture}.
\begin{figure}[h!]
\centering
\begin{tikzpicture}
[scale=.70]]

\draw[thick] (-7,0) --(-7,1)--(-6,1.5)--(-4.5,2.5)--(-2.5,1.5)--(-2.5,0)--(-3.5,0)--(-3.5,1)--(-4.5,1.5)--(-6,2.5)--(-8,1.5)--(-8,0)--(-7,0);
\draw[thick] (-7,1)--(-8,1.5);
\draw[thick] (-6,1.5)--(-6,2.5);
\draw[thick] (-4.5,1.5)--(-4.5,2.5);
\draw[thick] (-3.5,1)--(-2.5,1.5);

\draw[thick,dashed] (-7,0) --(-7,-1)--(-6,-1.5)--(-4.5,-1.5)--(-3.5,-1)--(-3.5,0)--(-2.5,0)--(-2.5,-1.5)--(-4.5,-2.5)--(-6,-2.5)--(-8,-1.5)--(-8,0)--(-7,0);
\draw[thick,dashed] (-5.25,-1.5) --(-5.25,-2.5);

\draw [fill] (-7,0)  circle [radius=0.1];
\draw [fill] (-7,1) circle [radius=0.1];
\draw [fill] (-6,1.5) circle [radius=0.1];
\draw [fill] (-4.5,2.5) circle [radius=0.1];
\draw [fill] (-2.5,1.5) circle [radius=0.1];
\draw [fill] (-2.5,0) circle [radius=0.1];
\draw [fill] (-3.5,0) circle [radius=0.1];
\draw [fill] (-3.5,1) circle [radius=0.1];
\draw [fill] (-4.5,1.5) circle [radius=0.1];
\draw [fill] (-6,2.5)circle [radius=0.1];
\draw [fill] (-8,1.5) circle [radius=0.1];
\draw [fill] (-8,0) circle [radius=0.1];

\draw [fill] (-5.25,-1.5) circle [radius=0.1];
\draw [fill] (-5.25,-2.5) circle [radius=0.1];

\node at (-6.3,0) {$x_{n-2}$};
\node at (-6.3,0.7) {$x_{n-1}$};
\node at (-5.7,1.2) {$x_{n}$};
\node at (-4.5,3) {$x_{n+1}$};
\node at (-1.5,1.5) {$x_{n+2}$};
\node at (-1.5,0) {$x_{n+3}$};
\node at (-4,0) {$x_{3}$};
\node at (-4,0.7) {$x_{2}$};
\node at (-4.5,1.2) {$x_{1}$};
\node at (-6,3) {$x_{2n}$};
\node at (-9,1.5) {$x_{2n-1}$};
\node at (-9,0) {$x_{2n-2}$};

\node at (-5.25,-1) {$x_{i}$};
\node at (-5.25,-3) {$x_{n+i}$};

\draw[thick] (3.5,0) --(3.5,1)--(4.5,1.5)--(6,1.5)--(7,1)--(7,0)--(8,0)--(8,1.5)--(6,2.5)--(4.5,2.5)--(2.5,1.5)--(2.5,0)--(3.5,0);
\draw[thick] (3.5,1)--(2.5,1.5);
\draw[thick] (4.5,1.5)--(4.5,2.5);
\draw[thick] (6,1.5)--(6,2.5);
\draw[thick] (7,1)--(8,1.5);

\draw[thick,dashed] (3.5,0) --(3.5,-1)--(4.5,-1.5)--(6,-1.5)--(7,-1)--(7,0)--(8,0)--(8,-1.5)--(6,-2.5)--(4.5,-2.5)--(2.5,-1.5)--(2.5,0)--(3.5,0);
\draw[thick,dashed] (5.25,-1.5) --(5.25,-2.5);

\draw [fill] (3.5,0)  circle [radius=0.1];
\draw [fill] (3.5,1) circle [radius=0.1];
\draw [fill] (4.5,1.5) circle [radius=0.1];
\draw [fill] (6,1.5) circle [radius=0.1];
\draw [fill] (7,1) circle [radius=0.1];
\draw [fill] (7,0) circle [radius=0.1];
\draw [fill] (8,0) circle [radius=0.1];
\draw [fill] (8,1.5) circle [radius=0.1];
\draw [fill] (6,2.5) circle [radius=0.1];
\draw [fill] (4.5,2.5)circle [radius=0.1];
\draw [fill] (2.5,1.5) circle [radius=0.1];
\draw [fill] (2.5,0) circle [radius=0.1];

\draw [fill] (5.25,-1.5) circle [radius=0.1];
\draw [fill] (5.25,-2.5) circle [radius=0.1];

\node at (4.3,0) {$x_{2n-5}$};
\node at (4.2,0.6) {$x_{2n-3}$};
\node at (4.9,1.1) {$x_{2n-1}$};
\node at (6,1.1) {$x_{1}$};
\node at (6.5,0.7) {$x_{3}$};
\node at (6.5,0) {$x_{5}$};
\node at (9,0) {$x_{n+5}$};
\node at (9,1.5) {$x_{n+3}$};
\node at (6,3) {$x_{n+1}$};
\node at (4.5,3) {$x_{n-1}$};
\node at (1.5,1.5) {$x_{n-3}$};
\node at (1.5,0) {$x_{n-5}$};

\node at (5.25,-1) {$x_{i}$};
\node at (5.25,-3) {$x_{n+i}$};

\end{tikzpicture}
\caption{The graphs $C_{2n}(1,n)$ and $C_{2n}(2,n)$.}\label{cubic-picture}
\end{figure}
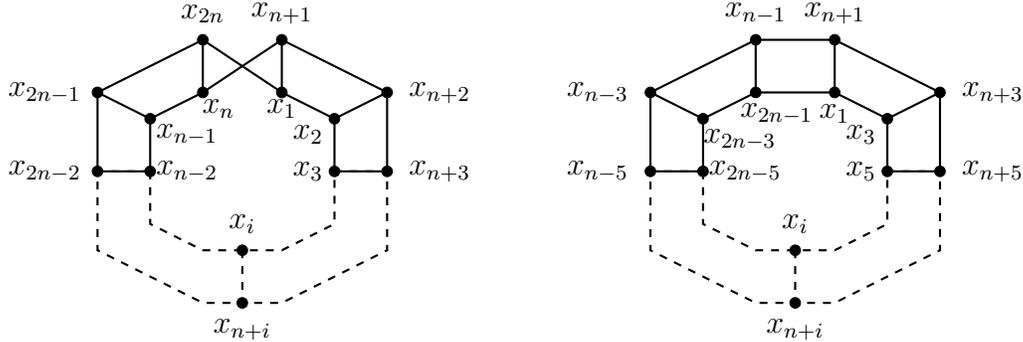

Our strategy is to use Theorem \ref{new-reg-result} to compute the
regularity of these two graphs.  Thus, we need bounds on 
${\rm reg}(I(G))$ and ${\rm pd}(I(G))$, and information about
the reduced Euler characteristic of ${\rm Ind}(G)$ when $G = C_{2n}(1,n)$
or $G_{2n}(2,n)$.  

We first bound the regularity and the projective dimension.  
We introducing the following three
families of graphs, where the $t\geq 1$ denotes the
number of ``squares'':
\begin{enumerate}
\item[$(i)$] The family $A_t$:
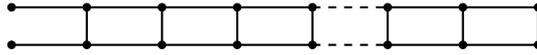
\[
\begin{tikzpicture}
[scale=.50]
\draw[thick] (-8,0.5) -- (-6,0.5) --(-4,0.5)--(-2,0.5)--(0,0.5)--(0,-0.5)--(-2,-0.5)--(-4,-0.5)--(-6,-0.5)-- (-6,0.5);
\draw[thick] (-8,-.5) -- (-6,-0.5);
\draw[thick] (-4,-0.5) --(-4,0.5);
\draw[thick] (-2,-0.5)--(-2,0.5);
\draw[thick,dashed] (0,.5)--(2,.5);
\draw[thick,dashed] (0,-.5)--(2,-.5);
\draw[thick] (2,0.5) --(4,0.5)--(6,0.5)--(6,-0.5)--(4,-0.5)--(2,-0.5)--(2,0.5);
\draw[thick] (4,-0.5) --(4,0.5);
 
\draw [fill] (-8,0.5) circle [radius=0.1];
\draw [fill] (-6,0.5) circle [radius=0.1];
\draw [fill] (-4,0.5) circle [radius=0.1];
\draw [fill] (-2,0.5) circle [radius=0.1];
\draw [fill] (0,0.5) circle [radius=0.1];
\draw [fill] (2,0.5) circle [radius=0.1];
\draw [fill] (4,0.5) circle [radius=0.1];
\draw [fill] (6,0.5) circle [radius=0.1];
\draw [fill] (-8,-0.5) circle [radius=0.1];
\draw [fill] (-6,-0.5) circle [radius=0.1];
\draw [fill] (-4,-0.5) circle [radius=0.1];
\draw [fill] (-2,-0.5) circle [radius=0.1];
\draw [fill] (0,-0.5) circle [radius=0.1];
\draw [fill] (2,-0.5) circle [radius=0.1];
\draw [fill] (4,-0.5) circle [radius=0.1];
\draw [fill] (6,-0.5) circle [radius=0.1];
\end{tikzpicture}\]
\item[$(ii)$]  The family $B_t$:
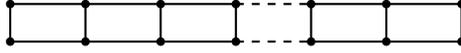
\[
\begin{tikzpicture}
[scale=.50]

\draw[thick] (-6,0.5) --(-4,0.5)--(-2,0.5)--(0,0.5)--(0,-0.5)--(-2,-0.5)--(-4,-0.5)--(-6,-0.5)-- (-6,0.5);
\draw[thick] (-4,-0.5) --(-4,0.5);
\draw[thick] (-2,-0.5)--(-2,0.5);

\draw[thick,dashed] (0,.5)--(2,.5);
\draw[thick,dashed] (0,-.5)--(2,-.5);

\draw[thick] (2,0.5) --(4,0.5)--(6,0.5)--(6,-0.5)--(4,-0.5)--(2,-0.5)--(2,0.5);
\draw[thick] (4,-0.5) --(4,0.5);

\draw [fill] (-6,0.5) circle [radius=0.1];
\draw [fill] (-4,0.5) circle [radius=0.1];
\draw [fill] (-2,0.5) circle [radius=0.1];
\draw [fill] (0,0.5) circle [radius=0.1];
\draw [fill] (2,0.5) circle [radius=0.1];
\draw [fill] (4,0.5) circle [radius=0.1];
\draw [fill] (6,0.5) circle [radius=0.1];

\draw [fill] (-6,-0.5) circle [radius=0.1];
\draw [fill] (-4,-0.5) circle [radius=0.1];
\draw [fill] (-2,-0.5) circle [radius=0.1];
\draw [fill] (0,-0.5) circle [radius=0.1];
\draw [fill] (2,-0.5) circle [radius=0.1];
\draw [fill] (4,-0.5) circle [radius=0.1];
\draw [fill] (6,-0.5) circle [radius=0.1];
\end{tikzpicture}
\]
\item[$(iii)$] The family $D_t$:
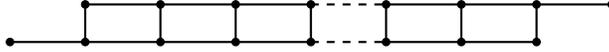
\[
\begin{tikzpicture}
[scale=.50]

\draw[thick] (-6,0.5) --(-4,0.5)--(-2,0.5)--(0,0.5)--(0,-0.5)--(-2,-0.5)--(-4,-0.5)--(-6,-0.5)-- (-6,0.5);
\draw[thick] (-8,-0.5) --(-6,-0.5);
\draw[thick] (6,0.5) --(8,0.5);

\draw[thick] (-4,-0.5) --(-4,0.5);
\draw[thick] (-2,-0.5)--(-2,0.5);

\draw[thick,dashed] (0,.5)--(2,.5);
\draw[thick,dashed] (0,-.5)--(2,-.5);

\draw[thick] (2,0.5) --(4,0.5)--(6,0.5)--(6,-0.5)--(4,-0.5)--(2,-0.5)--(2,0.5);
\draw[thick] (4,-0.5) --(4,0.5);

\draw [fill] (-6,0.5) circle [radius=0.1];
\draw [fill] (-4,0.5) circle [radius=0.1];
\draw [fill] (-2,0.5) circle [radius=0.1];
\draw [fill] (0,0.5) circle [radius=0.1];
\draw [fill] (2,0.5) circle [radius=0.1];
\draw [fill] (4,0.5) circle [radius=0.1];
\draw [fill] (6,0.5) circle [radius=0.1];

\draw [fill] (-8,-0.5) circle [radius=0.1];
\draw [fill] (-6,-0.5) circle [radius=0.1];
\draw [fill] (-4,-0.5) circle [radius=0.1];
\draw [fill] (-2,-0.5) circle [radius=0.1];
\draw [fill] (0,-0.5) circle [radius=0.1];
\draw [fill] (2,-0.5) circle [radius=0.1];
\draw [fill] (4,-0.5) circle [radius=0.1];
\draw [fill] (6,-0.5) circle [radius=0.1];
\draw [fill] (8,0.5) circle [radius=0.1];
\end{tikzpicture}
\]
\end{enumerate}

\begin{lemma} 
\label{projective-bounds} 
With the notation as above, we have
\begin{enumerate}
\item[$(i)$]
If $G = A_t$, then
\[{\rm reg}(I(G)) \leq
\begin{cases}
\frac{t+4}{2} & \mbox{if $t$ even} \\
\frac{t+3}{2} & \mbox{if $t$ odd}
\end{cases}
~~\mbox{and}~~~
{\rm pd}(I(G)) \leq
\begin{cases}
\frac{3t}{2}+1 & \mbox{if $t$ even} \\
\frac{3(t-1)}{2}+2 & \mbox{if $t$ odd.}
\end{cases}\]
\item[$(ii)$]
If $G = B_t$, then
\[{\rm reg}(I(G)) \leq
\begin{cases}
\frac{t+4}{2} & \mbox{if $t$ even} \\
\frac{t+3}{2} & \mbox{if $t$ odd.}
\end{cases}\]
\item[$(iii)$]
If $G = D_t$ and $t=2l+1$ with $l$ an odd number, then ${\rm reg}(I(G)) \leq \frac{t+3}{2}$.
\end{enumerate}

\end{lemma}

\begin{proof}
$(i)$ The proof is by induction on $t$.  Via a direct computation
(for example, using {\it Macaulay2}), one finds ${\rm reg}(I(A_1)) = 2$,
${\rm reg}(I(A_2)) = 3$, ${\rm pd}(I(A_1)) = 2$, and ${\rm pd}(I(A_2)) =4$.
Our values agree with the upper bounds given in
the statement, so the base cases hold. 

Now suppose that $t \geq 3$.  The graph $A_t$
can be decomposed into the subgraphs $A_{1}$ and $A_{t-2}$, i.e.,
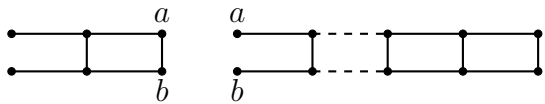
\[
\begin{tikzpicture}
[scale=.50]

\draw[thick] (-8,0.5) -- (-6,0.5) --(-4,0.5); 
\draw[thick] (-2,0.5)--(0,0.5)--(0,-0.5)--(-2,-0.5); 
\draw[thick] (-4,-0.5)--(-6,-0.5)-- (-6,0.5);
\draw[thick] (-8,-.5) -- (-6,-0.5);
\draw[thick] (-4,-0.5) --(-4,0.5);

\draw[thick,dashed] (0,.5)--(2,.5);
\draw[thick,dashed] (0,-.5)--(2,-.5);

\draw[thick] (2,0.5) --(4,0.5)--(6,0.5)--(6,-0.5)--(4,-0.5)--(2,-0.5)--(2,0.5);
\draw[thick] (4,-0.5) --(4,0.5);

\draw [fill] (-8,0.5) circle [radius=0.1];
\draw [fill] (-6,0.5) circle [radius=0.1];
\draw [fill] (-4,0.5) circle [radius=0.1];
\draw [fill] (-2,0.5) circle [radius=0.1];
\draw [fill] (0,0.5) circle [radius=0.1];
\draw [fill] (2,0.5) circle [radius=0.1];
\draw [fill] (4,0.5) circle [radius=0.1];
\draw [fill] (6,0.5) circle [radius=0.1];

\draw [fill] (-8,-0.5) circle [radius=0.1];
\draw [fill] (-6,-0.5) circle [radius=0.1];
\draw [fill] (-4,-0.5) circle [radius=0.1];
\draw [fill] (-2,-0.5) circle [radius=0.1];
\draw [fill] (0,-0.5) circle [radius=0.1];
\draw [fill] (2,-0.5) circle [radius=0.1];
\draw [fill] (4,-0.5) circle [radius=0.1];
\draw [fill] (6,-0.5) circle [radius=0.1];

\node at (-4,1) {$a$};
\node at (-2,1) {$a$};

\node at (-4,-1) {$b$};
\node at (-2,-1) {$b$};
\end{tikzpicture}\]
Suppose that $t$ is even.  By Theorem \ref{KM} and by induction
(and the fact that ${\rm reg}(R/I) ={\rm reg}(I) -1$)
we get
\[
{\rm reg}(R/I(A_t)) \leq {\rm reg}(R/I(A_1)) + {\rm reg}(R/I(A_{t-2})) 
\leq 1 + \frac{(t-2)+4}{2}-1 = \frac{t+4}{2}-1
\]
and 
\[
{\rm pd}(I(A_t)) \leq {\rm pd}(I(A_1)) + {\rm pd}(I(A_{t-2}) + 1 \leq 
2 + \frac{3(t-2)}{2} +1 + 1 =  \frac{3t}{2}+1.\]
Because the proof for when $t$ is odd is similar, we omit it.

$(ii)$ 
A direct computation shows ${\rm reg}(I(B_1)) = 2$ and 
${\rm reg}(I(B_2)) = 3$.  If $t \geq 3$, we decompose $B_t$ 
into the subgraphs $B_{1}$ and $A_{t-2}$, i.e.,
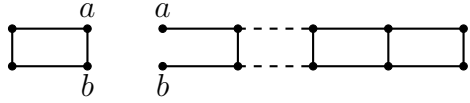
\[
\begin{tikzpicture}
[scale=.50]

\draw[thick] (-6,0.5) --(-4,0.5); 
\draw[thick] (-2,0.5)--(0,0.5)--(0,-0.5)--(-2,-0.5); 
\draw[thick] (-4,-0.5)--(-6,-0.5)-- (-6,0.5);
\draw[thick] (-4,-0.5) --(-4,0.5);

\draw[thick,dashed] (0,.5)--(2,.5);
\draw[thick,dashed] (0,-.5)--(2,-.5);

\draw[thick] (2,0.5) --(4,0.5)--(6,0.5)--(6,-0.5)--(4,-0.5)--(2,-0.5)--(2,0.5);
\draw[thick] (4,-0.5) --(4,0.5);

\draw [fill] (-6,0.5) circle [radius=0.1];
\draw [fill] (-4,0.5) circle [radius=0.1];
\draw [fill] (-2,0.5) circle [radius=0.1];
\draw [fill] (0,0.5) circle [radius=0.1];
\draw [fill] (2,0.5) circle [radius=0.1];
\draw [fill] (4,0.5) circle [radius=0.1];
\draw [fill] (6,0.5) circle [radius=0.1];

\draw [fill] (-6,-0.5) circle [radius=0.1];
\draw [fill] (-4,-0.5) circle [radius=0.1];
\draw [fill] (-2,-0.5) circle [radius=0.1];
\draw [fill] (0,-0.5) circle [radius=0.1];
\draw [fill] (2,-0.5) circle [radius=0.1];
\draw [fill] (4,-0.5) circle [radius=0.1];
\draw [fill] (6,-0.5) circle [radius=0.1];

\node at (-4,1) {$a$};
\node at (-2,1) {$a$};

\node at (-4,-1) {$b$};
\node at (-2,-1) {$b$};
\end{tikzpicture}\]
Suppose that $t$ is even. 
Since ${\rm reg}(I(B_1)) =2$, Theorem~\ref{KM} and part $(i)$ above gives us: 
\[
{\rm reg}(R/I(B_t)) \leq {\rm reg}(R/I(B_1)) + {\rm reg}(R/I(A_{t-2})) \leq 
\frac{(t-2)+4}{2} =\frac{t+2}{2}.
\]
Therefore ${\rm reg}(I(B_t)) \leq \frac{t+2}{2}+1=\frac{t+4}{2}$. When $t$ is odd, the proof is similar. 

$(iii)$ 
Because $t=2l+1$ with $l$ odd, the
graph $D_t$ can be decomposed into $l+1$ subgraphs of the form
$A_{1}$, i.e.,
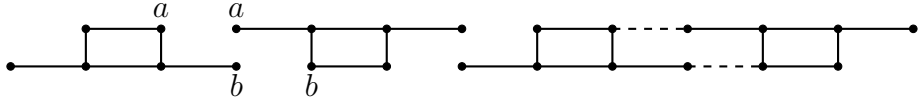
\[
\begin{tikzpicture}
[scale=.50]
\draw[thick] (-8,0.5)--(-6,0.5);
\draw[thick] (-10,-0.5) --(-8,-0.5)--(-6,-0.5)--(-4,-0.5);
\draw[thick] (-8,-0.5) --(-8,0.5);
\draw[thick] (-6,-0.5) --(-6,0.5);
\node at (-6,1) {$a$};
\node at (-4,1) {$a$};
\node at (-4,-1) {$b$};
\node at (-2,-1) {$b$};

\draw[thick] (-2,-0.5)--(0,-0.5);
\draw[thick] (-4,0.5) --(-2,0.5)--(0,0.5)--(2,0.5);
\draw[thick] (-2,-0.5) --(-2,0.5);
\draw[thick] (0,-0.5) --(0,0.5);

\draw[thick] (4,0.5)--(6,0.5);
\draw[thick] (2,-0.5) --(4,-0.5)--(6,-0.5)--(8,-0.5);
\draw[thick] (4,-0.5) --(4,0.5);
\draw[thick] (6,-0.5) --(6,0.5);

\draw[thick] (10,-0.5)--(12,-0.5);
\draw[thick] (8,0.5) --(10,0.5)--(12,0.5)--(14,0.5);
\draw[thick] (10,-0.5) --(10,0.5);
\draw[thick] (12,-0.5) --(12,0.5);


\draw[thick,dashed] (6,.5)--(8,.5);
\draw[thick,dashed] (8,-.5)--(10,-.5);


\draw [fill] (-8,0.5) circle [radius=0.1];
\draw [fill] (-6,0.5) circle [radius=0.1];
\draw [fill] (-4,0.5) circle [radius=0.1];
\draw [fill] (-2,0.5) circle [radius=0.1];
\draw [fill] (0,0.5) circle [radius=0.1];
\draw [fill] (2,0.5) circle [radius=0.1];
\draw [fill] (4,0.5) circle [radius=0.1];
\draw [fill] (6,0.5) circle [radius=0.1];
\draw [fill] (8,0.5) circle [radius=0.1];
\draw [fill] (10,0.5) circle [radius=0.1];
\draw [fill] (12,0.5) circle [radius=0.1];
\draw [fill] (14,0.5) circle [radius=0.1];

\draw [fill] (-10,-0.5) circle [radius=0.1];
\draw [fill] (-8,-0.5) circle [radius=0.1];
\draw [fill] (-6,-0.5) circle [radius=0.1];
\draw [fill] (-4,-0.5) circle [radius=0.1];
\draw [fill] (-2,-0.5) circle [radius=0.1];
\draw [fill] (0,-0.5) circle [radius=0.1];
\draw [fill] (2,-0.5) circle [radius=0.1];
\draw [fill] (4,-0.5) circle [radius=0.1];
\draw [fill] (6,-0.5) circle [radius=0.1];
\draw [fill] (8,-0.5) circle [radius=0.1];
\draw [fill] (10,-0.5) circle [radius=0.1];
\draw [fill] (12,-0.5) circle [radius=0.1];


\end{tikzpicture}
\]
Since ${\rm reg}(I(A_1)) = 2$, by Theorem~\ref{KM} we get  
${\rm reg}(R/I(D_t)) \leq (l+1){\rm reg}(R/I(A_1)) = l+1$.
Thus ${\rm reg}(I(D_t)) \leq l+2=\frac{t+3}{2}$.
\end{proof}

We now bound the projective dimensions
of the edge ideals of $C_{2n}(1,n)$ and $C_{2n}(2,n)$.

\begin{lemma}\label{pdim-bounds}  Let $n \geq 4$.
\hspace{.1cm}\vspace{.1cm}

\begin{enumerate}
\item[$(i)$]
If $G = C_{2n}(1,n)$, then 
\[{\rm pd}(I(G)) \leq \begin{cases}
3k-1 & \mbox{if $n = 2k$} \\
3k+1 & \mbox{if $n= 2k+1$.} 
\end{cases}
\]
\item[$(ii)$]
If $G = C_{2n}(2,n)$, then ${\rm pd}(I(G)) \leq 3k+1$ where $n=2k+1$.
\end{enumerate}
\end{lemma}

\begin{proof}
$(i)$ Let $G = C_{2n}(1,n)$, suppose that $n= 2k+1$. 
The graph $C_{2n}(1,n)$ can be decomposed 
into the subgraphs $A_{1}$ and $A_{2k-2}$, i.e.,
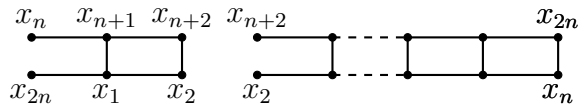
\[
\begin{tikzpicture}
[scale=.50]

\draw[thick] (-8,0.5) -- (-6,0.5) --(-4,0.5); 
\draw[thick] (-2,0.5)--(0,0.5)--(0,-0.5)--(-2,-0.5); 
\draw[thick] (-4,-0.5)--(-6,-0.5)-- (-6,0.5);
\draw[thick] (-8,-.5) -- (-6,-0.5);
\draw[thick] (-4,-0.5) --(-4,0.5);

\draw[thick,dashed] (0,.5)--(2,.5);
\draw[thick,dashed] (0,-.5)--(2,-.5);

\draw[thick] (2,0.5) --(4,0.5)--(6,0.5)--(6,-0.5)--(4,-0.5)--(2,-0.5)--(2,0.5);
\draw[thick] (4,-0.5) --(4,0.5);

\draw [fill] (-8,0.5) circle [radius=0.1];
\draw [fill] (-6,0.5) circle [radius=0.1];
\draw [fill] (-4,0.5) circle [radius=0.1];
\draw [fill] (-2,0.5) circle [radius=0.1];
\draw [fill] (0,0.5) circle [radius=0.1];
\draw [fill] (2,0.5) circle [radius=0.1];
\draw [fill] (4,0.5) circle [radius=0.1];
\draw [fill] (6,0.5) circle [radius=0.1];

\draw [fill] (-8,-0.5) circle [radius=0.1];
\draw [fill] (-6,-0.5) circle [radius=0.1];
\draw [fill] (-4,-0.5) circle [radius=0.1];
\draw [fill] (-2,-0.5) circle [radius=0.1];
\draw [fill] (0,-0.5) circle [radius=0.1];
\draw [fill] (2,-0.5) circle [radius=0.1];
\draw [fill] (4,-0.5) circle [radius=0.1];
\draw [fill] (6,-0.5) circle [radius=0.1];

\node at (-8,1) {$x_{n}$};
\node at (-8,-1) {$x_{2n}$};
\node at (-6,1) {$x_{n+1}$};
\node at (-6,-1) {$x_{1}$};

\node at (6,1) {$x_{2n}$};
\node at (6,-1) {$x_{n}$};

\node at (-4,1) {$x_{n+2}$};
\node at (-2,1) {$x_{n+2}$};

\node at (-4,-1) {$x_2$};
\node at (-2,-1) {$x_2$};
\node at (6,1) {$x_{2n}$};
\node at (6,-1) {$x_{n}$};

\end{tikzpicture}
\]
Note that since $n \geq 4$ and $n$ odd, $2k -2 \geq 2$.
Combining Theorem~\ref{KM} and Lemma~\ref{projective-bounds} we get: 
\[
{\rm pd}(I(C_{2n}(1,n))) \leq {\rm pd}(I(A_{2k-2}))+{\rm pd}(I(A_1))+1 \leq 
\left(\frac{3(2k-2)}{2}+1\right)+3=3k+1.
\]

If $n=2k$, $C_{2n}(1,n)$ can be decomposed 
as in the previous case with the only difference 
being that $C_{2n}(1,n)$ can be decomposed 
into the union of the subgraphs $A_{1}$ and $A_{2k-3}$. 
By Theorem~\ref{KM} and Lemma~\ref{projective-bounds}:  
\[
{\rm pd}(I(C_{2n}(1,n))) \leq {\rm pd}(I(A_{2k-3}))+{\rm pd}(I(A_1))+1 \leq 
\left(\frac{3(2k-4)}{2}+2\right)+3=3k-1.
\]

$(ii)$ Let $G = C_{2n}(2,n)$ with $n= 2k+1$. We can draw $G$ 
as
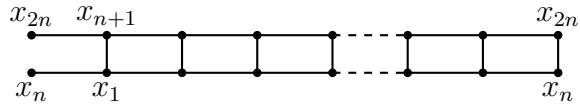
\[
\begin{tikzpicture}
[scale=.50]

\draw[thick] (-8,0.5) -- (-6,0.5) --(-4,0.5)--(-2,0.5)--(0,0.5)--(0,-0.5)--(-2,-0.5)--(-4,-0.5)--(-6,-0.5)-- (-6,0.5);
\draw[thick] (-8,-.5) -- (-6,-0.5);
\draw[thick] (-4,-0.5) --(-4,0.5);
\draw[thick] (-2,-0.5)--(-2,0.5);

\draw[thick,dashed] (0,.5)--(2,.5);
\draw[thick,dashed] (0,-.5)--(2,-.5);

\draw[thick] (2,0.5) --(4,0.5)--(6,0.5)--(6,-0.5)--(4,-0.5)--(2,-0.5)--(2,0.5);
\draw[thick] (4,-0.5) --(4,0.5);

\draw [fill] (-8,0.5) circle [radius=0.1];
\draw [fill] (-6,0.5) circle [radius=0.1];
\draw [fill] (-4,0.5) circle [radius=0.1];
\draw [fill] (-2,0.5) circle [radius=0.1];
\draw [fill] (0,0.5) circle [radius=0.1];
\draw [fill] (2,0.5) circle [radius=0.1];
\draw [fill] (4,0.5) circle [radius=0.1];
\draw [fill] (6,0.5) circle [radius=0.1];

\draw [fill] (-8,-0.5) circle [radius=0.1];
\draw [fill] (-6,-0.5) circle [radius=0.1];
\draw [fill] (-4,-0.5) circle [radius=0.1];
\draw [fill] (-2,-0.5) circle [radius=0.1];
\draw [fill] (0,-0.5) circle [radius=0.1];
\draw [fill] (2,-0.5) circle [radius=0.1];
\draw [fill] (4,-0.5) circle [radius=0.1];
\draw [fill] (6,-0.5) circle [radius=0.1];
\node at (-8,1) {$x_{2n}$};
\node at (-8,-1) {$x_{n}$};
\node at (-6,1) {$x_{n+1}$};
\node at (-6,-1) {$x_{1}$};

\node at (6,1) {$x_{2n}$};
\node at (6,-1) {$x_{n}$};

\end{tikzpicture}
\]
The previous representation of $G$ contains $2k$ squares. 
Then the graph $G$ can be decomposed 
into the subgraphs $A_{1}$ and $A_{2k-2}$, and 
the proof runs as in $(i)$.
\end{proof}

We now determine bounds on the regularity.

\begin{lemma}\label{reg-bounds}Let $n \geq 4$.
\hspace{.1cm}\vspace{.1cm}

\begin{enumerate}
\item[$(i)$]
If $G = C_{2n}(1,n)$, then 
\[{\rm reg}(I(G)) \leq \begin{cases}
k+1 & \mbox{if $n = 2k$, or if $n=2k+1$ and $k$ odd} \\
k+2 & \mbox{if $n= 2k+1$ and $k$ even.} 
\end{cases}
\]
\item[$(ii)$]
If $G = C_{2n}(2,n)$, then 
\[{\rm reg}(I(G)) \leq \begin{cases}
k+1 & \mbox{if $n=2k+1$ and $k$ even} \\
k+2 & \mbox{if $n=2k+1$ and $k$ odd.} 
\end{cases}
\]
\end{enumerate}
\end{lemma}

\begin{proof}
$(i)$ Let $G = C_{2n}(1,n)$.   We now consider three cases.

\noindent
{\it Case 1.} $n = 2k$.
\noindent

In Lemma~\ref{pdim-bounds} $(i)$ we saw that $G$ can be decomposed into the subgraphs $A_{1}$ and $A_{2k-3}$. 
By Theorem~\ref{KM} and Lemma~\ref{projective-bounds} we get:
\[
{\rm reg}(R/I(G)) \leq {\rm reg}(R/I(A_1)) + {\rm reg}(R/I(A_{2k-3})) \leq 
k.\]

\noindent
{\it Case 2.} $n=2k+1$ with $k$ an odd number.
\noindent

Using Lemma~\ref{reg-results} $(v)$, we have: 
\[
{\rm reg}(I(G)) \in \{{\rm reg}(I(G\setminus x_1),{\rm reg}(I(G\setminus N[x_1])+1\}.\]
If we set $W = G \setminus x_1$, then by applying Lemma \ref{reg-results}
$(v)$ again, we have
\[ 
{\rm reg}(I(G)) \in \{{\rm reg}(I(W\setminus x_{n+1}),{\rm reg}(I(W\setminus 
N[x_{n+1}])+1,
{\rm reg}(I(G\setminus N[x_1])+1\}.
\]

We have $G \setminus N[x_1] \cong W \setminus N[x_{n+1}] \cong 
D_{2k-3}$. Moreover, $2k-3=2(k-2)+1$, and since $k$ is an odd number, 
$k-2$ is also odd.  Thus by Lemma~\ref{projective-bounds} $(iii)$ 
we obtain ${\rm reg}(I(D_{2k-3}))\leq \frac{2k-3+3}{2}=k$.  
On the other hand, the graph $W \setminus x_{n+1} =
(G\setminus x_1) \setminus x_{n+1} \cong B_{2k-1}$, so 
by Lemma~\ref{projective-bounds} $(ii)$ we have 
${\rm reg}(I(W \setminus x_{n+1}))\leq \frac{2k-1+3}{2} \leq k+1$.
Thus, ${\rm reg}(I(G)) \leq k+1$.

\noindent
{\it Case 3.} $n=2k+1$ with $k$ an even number.
\noindent 

In Lemma~\ref{pdim-bounds} $(i)$ we saw that $G$ can be 
decomposed into the subgraphs $A_{1}$ and $A_{2k-2}$, and the 
proof runs as in Case 1.

$(ii)$ Let $G = C_{2n}(2,n)$. We consider two cases.

\noindent
{\it Case 1.} $n = 2k+1$ with $k$ an even number. 
\noindent

As in the second case of $(i)$, by Lemma~\ref{reg-results} $(v)$ we
have 
\[ 
{\rm reg}(I(G)) \in \{{\rm reg}(I(W\setminus x_{n+1}),{\rm reg}(I(W\setminus 
N[x_{n+1}])+1,
{\rm reg}(I(G\setminus N[x_1])+1\}.
\]
where $W = G\setminus x_1$.   In particular, $W \setminus N[x_{n+1}]
\cong G \setminus N[x_1]$.   The graph $G\setminus N[x_1]$ 
can be represented as
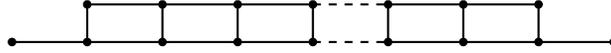
\[
\begin{tikzpicture}
[scale=.50]

\draw[thick] (-6,0.5) --(-4,0.5)--(-2,0.5)--(0,0.5)--(0,-0.5)--(-2,-0.5)--(-4,-0.5)--(-6,-0.5)-- (-6,0.5);
\draw[thick] (-8,-0.5) --(-6,-0.5);
\draw[thick] (6,-0.5) --(8,-0.5);

\draw[thick] (-4,-0.5) --(-4,0.5);
\draw[thick] (-2,-0.5)--(-2,0.5);

\draw[thick,dashed] (0,.5)--(2,.5);
\draw[thick,dashed] (0,-.5)--(2,-.5);

\draw[thick] (2,0.5) --(4,0.5)--(6,0.5)--(6,-0.5)--(4,-0.5)--(2,-0.5)--(2,0.5);
\draw[thick] (4,-0.5) --(4,0.5);

\draw [fill] (-6,0.5) circle [radius=0.1];
\draw [fill] (-4,0.5) circle [radius=0.1];
\draw [fill] (-2,0.5) circle [radius=0.1];
\draw [fill] (0,0.5) circle [radius=0.1];
\draw [fill] (2,0.5) circle [radius=0.1];
\draw [fill] (4,0.5) circle [radius=0.1];
\draw [fill] (6,0.5) circle [radius=0.1];

\draw [fill] (-8,-0.5) circle [radius=0.1];
\draw [fill] (-6,-0.5) circle [radius=0.1];
\draw [fill] (-4,-0.5) circle [radius=0.1];
\draw [fill] (-2,-0.5) circle [radius=0.1];
\draw [fill] (0,-0.5) circle [radius=0.1];
\draw [fill] (2,-0.5) circle [radius=0.1];
\draw [fill] (4,-0.5) circle [radius=0.1];
\draw [fill] (6,-0.5) circle [radius=0.1];
\draw [fill] (8,-0.5) circle [radius=0.1];

\end{tikzpicture}
\]
The previous representation of $G\setminus N[x_1]$ contains $2k-3$ squares. 
It follows that   $G\setminus N[x_1]$ can be decomposed 
into the subgraphs $D_{2k-5}$ and $A_{1}$, i.e.,
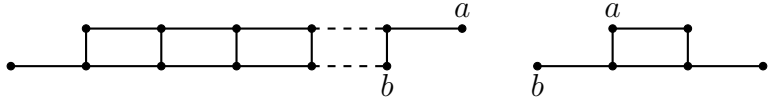
\[
\begin{tikzpicture}
[scale=.50]

\draw[thick] (-6,0.5) --(-4,0.5)--(-2,0.5)--(0,0.5)--(0,-0.5)--(-2,-0.5)--(-4,-0.5)--(-6,-0.5)-- (-6,0.5);
\draw[thick] (-8,-0.5) --(-6,-0.5);

\draw[thick] (-4,-0.5) --(-4,0.5);
\draw[thick] (-2,-0.5)--(-2,0.5);

\draw[thick,dashed] (0,.5)--(2,.5);
\draw[thick,dashed] (0,-.5)--(2,-.5);

\draw[thick] (2,0.5) --(2,-0.5);
\draw[thick] (2,0.5) --(4,0.5);

\draw[thick] (8,0.5) --(10,0.5)--(10,-0.5)--(8,-0.5)--(8,0.5);
\draw[thick] (10,-0.5) --(12,-0.5);
\draw[thick] (8,-0.5)--(6,-0.5);

\draw [fill] (-6,0.5) circle [radius=0.1];
\draw [fill] (-4,0.5) circle [radius=0.1];
\draw [fill] (-2,0.5) circle [radius=0.1];
\draw [fill] (0,0.5) circle [radius=0.1];
\draw [fill] (2,0.5) circle [radius=0.1];
\draw [fill] (4,0.5) circle [radius=0.1];

\draw [fill] (-8,-0.5) circle [radius=0.1];
\draw [fill] (-6,-0.5) circle [radius=0.1];
\draw [fill] (-4,-0.5) circle [radius=0.1];
\draw [fill] (-2,-0.5) circle [radius=0.1];
\draw [fill] (0,-0.5) circle [radius=0.1];
\draw [fill] (2,-0.5) circle [radius=0.1];

\draw [fill] (8,0.5) circle [radius=0.1];
\draw [fill] (10,0.5) circle [radius=0.1];
\draw [fill] (10,-0.5) circle [radius=0.1];
\draw [fill] (8,-0.5) circle [radius=0.1];
\draw [fill] (6,-0.5) circle [radius=0.1];
\draw [fill] (12,-0.5) circle [radius=0.1];

\node at (4,1) {$a$};
\node at (2,-1) {$b$};

\node at (8,1) {$a$};
\node at (6,-1) {$b$};

\end{tikzpicture}
\]
Note that $2k-5=2(k-3)+1$, and because $k$ is even, then $k-3$ is odd. 
Using Theorem~\ref{KM} and 
Lemma~\ref{projective-bounds} we get:
\[{\rm reg}(R/I(G\setminus N[x_1])) \leq {\rm reg}(R/I(D_{2k-5})) + {\rm reg}(R/I(A_{1})) \leq \frac{2k-2}{2}=k-1.\]

The graph $W \setminus x_{n+1} \cong B_{2k-1}$.  So 
by Lemma~\ref{projective-bounds} $(ii)$ we have 
${\rm reg}(I(W \setminus x_{n+1}))\leq \frac{2k-1+3}{2}=k+1$.
Consequently,  ${\rm reg}(I(G)) \leq k+1$, as desired.

\noindent
{\it Case 2.} $n = 2k+1$ with $k$ an odd number. 
\noindent 

The result follows from the fact that 
the graphs $C_{2n}(2,n)$ 
can be decomposed into the subgraphs $A_{1}$ and $A_{2k-2}$,
and so ${\rm reg}(I(G)) \leq {\rm reg}(I(A_1)) + {\rm reg}(I(A_{2k-2}))-1$.
\end{proof}

Our final ingredient is a 
result of Hoshino \cite[Theorem 2.26]{RH} (also see Brown-Hoshino 
\cite[Theorem 3.5]{BH}) which
describes the independence polynomial 
for cubic circulant graphs.

\begin{theorem}\label{independencepoly}
For each $n \geq 2$, set 
\[I_n(x) = 1 + \sum_{\ell=0}^{\lfloor \frac{n-2}{4} \rfloor} \frac{2n}{2\ell+1}
\binom{n-2\ell-2}{2\ell}x^{2\ell+1}(1+x)^{n-4\ell+2}.\]
\begin{enumerate}
\item[$(i)$]  If $G = C_{2n}(1,n)$ with $n$ even, or if $G = C_{2n}(2,n)$ with
$n$ odd, then
$I(G,x) = I_n(x)$.
\item[$(ii)$] If $G = C_{2n}(1,n)$ and $n$ odd, then $I(G,x) = I_n(x) + 2x^n$.
\end{enumerate}
\end{theorem}

We now come to the main result of this section.

\begin{theorem}\label{maintheorem2}
Let $1 \leq a < n$  and $t={\rm gcd}(2n,a)$. 
\begin{enumerate}
\item[$(a)$] If $\frac{2n}{t}$ is even, then: 
\[
{\rm reg}(I(C_{2n}(a,n))) = \begin{cases}
kt+1 & \mbox{if $\frac{n}{t}=2k$, or $\frac{n}{t}=2k+1$ with $k$ an odd number} \\
(k+1)t+1 & \mbox{if $\frac{n}{t}=2k+1$ with $k$ an even number.} 
\end{cases} 
\]
\item[$(b)$] If $\frac{2n}{t}$ is odd, then: 
\[
{\rm reg}(I(C_{2n}(a,n))) = \begin{cases}
\frac{kt}{2}+1 & \mbox{if $\frac{2n}{t}=2k+1$ with $k$ an even number} \\
\frac{(k+1)t}{2}+1 & \mbox{if $\frac{2n}{t}=2k+1$ with $k$ an odd number.} 
\end{cases}
\]
\end{enumerate}
\end{theorem}

\begin{proof} The formulas can verified directly for the special
cases that $n=2$ (i.e., $G = C_4(1,2)$) or $n=3$ (i.e., $G = C_6(1,3)$
and $C_6(2,3)$).  We can therefore assume $n \geq 4$. 
In light of Theorem \ref{isomorphic-a-n} and 
Lemma \ref{reg-results} $(i)$ it will 
suffice to prove that the inequalities of Lemma \ref{reg-bounds} are actually
equalities.  We will make use Theorem \ref{new-reg-result}. 
We consider five cases, where the proof of each case is similar.

\noindent
{\it Case 1.} $G = C_{2n}(1,n)$ with $n=2k$.
\noindent

In this case, Lemma \ref{pdim-bounds} gives ${\rm pd}(I(G)) \leq 3k-1$,
Lemma \ref{reg-bounds} gives ${\rm reg}(I(G)) \leq k+1$. 
Furthermore, since $\widetilde{\chi}({\rm Ind}(G)) = -I(G,-1)$
by equation \eqref{eq-Euler-independence},
Theorem \ref{independencepoly} gives $\widetilde{\chi}({\rm Ind}(G)) = -1$.
Because $G$ has $4k = (k+1)+(3k-1)$ vertices, Theorem
\ref{new-reg-result} $(ii)$ implies ${\rm reg}(I(G)) = k+1$.

\noindent
{\it Case 2.} $G = C_{2n}(1,n)$ with $n=2k+1$ and $k$ even.

We have
${\rm reg}(I(G)) \leq k+2$ and ${\rm pd}(I(G)) \leq 3k+1
=(4k+2)-(k+2)+1 = n - (k+2)+1$ 
by Lemmas \ref{pdim-bounds} and \ref{reg-bounds}, respectively.
Because $n$ is odd, $\widetilde{\chi}({\rm Ind}(G)) = -[I_n(-1) + 2(-1)^n]
=-[1-2]=1 >0$.  So, ${\rm reg}(I(G)) = k+2$ by Theorem
\ref{new-reg-result} $(i)$ $(a)$ because $k+2$ is even and 
$\widetilde{\chi}({\rm Ind}(G)) > 0$.

\noindent
{\it Case 3.} $G = C_{2n}(1,n)$ with $n=2k+1$ and $k$ odd.

We have ${\rm reg}(I(G)) = k+1$ 
by Theorem \ref{new-reg-result} $(ii)$ because
${\rm reg}(I(G)) \leq k+1$ (Lemma \ref{reg-bounds}), ${\rm pd}(I(G)) \leq
3k+1$ (Lemma \ref{pdim-bounds}), $2n=4k+2$ is the number of variables, and
$\widetilde{\chi}({\rm Ind}(G)) \neq 0$.  

\noindent
{\it Case 4.} $G = C_{2n}(2,n)$ with $n=2k+1$ and $k$ even.

We have ${\rm reg}(I(G)) = k+1$ from Theorem
\ref{new-reg-result} $(ii)$ since ${\rm reg}(I(G)) \leq k+1$
(Lemma \ref{reg-bounds})
${\rm pd}(I(G)) \leq 3k+1$ (Lemma \ref{pdim-bounds}), 
and $\widetilde{\chi}({\rm Ind}(G)) = -I(G,-1) = -1 \neq 0$ 
(Theorem \ref{independencepoly}). 

\noindent
{\it Case 5.} $G = C_{2n}(2,n)$ with $n=2k+1$ and $k$ odd.

In our final case, ${\rm reg}(I(G)) \leq k+2$ by Lemma \ref{reg-bounds},
${\rm pd}(I(G)) \leq 3k+1$ by Lemma \ref{pdim-bounds}.  Since
$n$ is odd, $\widetilde{\chi}({\rm Ind}(G)) = -I(G,-1) = -1$ by
Theorem \ref{independencepoly}.  Since $k$ is odd, $k+2$ is
odd.  Because $2n=4k+2$ is the number of variables, we have
${\rm reg}(I(G)) = k+2$ by Theorem \ref{new-reg-result} $(i)$ $(b)$. 

These five cases now complete the proof.
\end{proof}


\bibliographystyle{plain}

\end{document}